\documentclass{amsart}
\usepackage{amssymb
,amsthm
,amsmath
,mathtools
}
\usepackage[all]{xy}
\usepackage[top=30truemm,bottom=30truemm,left=25truemm,right=25truemm]{geometry}

\allowdisplaybreaks

\newcommand{\Z}{\mathbb{Z}}
\newcommand{\C}{\mathbb{C}}
\newcommand{\BB}{\mathcal{B}}
\newcommand{\VV}{\mathcal{V}}
\newcommand{\Sch}{\mathcal{S}}
\newcommand{\WW}{\mathcal{W}}
\newcommand{\cent}{\mathrm{Cent}}
\newcommand{\Norm}{\mathrm{Norm}}
\newcommand{\id}{\mathrm{id}}
\newcommand{\tr}{\mathrm{tr}}
\newcommand{\im}{\mathrm{Im}}
\newcommand{\re}{\mathrm{Re}}
\newcommand{\Irr}{\mathrm{Irr}}
\newcommand{\temp}{\mathrm{temp}}
\newcommand{\Ind}{\mathrm{Ind}}
\newcommand{\Hom}{\mathrm{Hom}}
\newcommand{\Sp}{\mathrm{Sp}}
\newcommand{\SL}{\mathrm{SL}}
\newcommand{\GL}{\mathrm{GL}}
\newcommand{\U}{\mathrm{U}}
\newcommand{\ep}{\varepsilon}
\newcommand{\orth}{\mathrm{O}}
\newcommand{\SO}{\mathrm{SO}}
\newcommand{\WD}{{\it WD}}
\newcommand{\half}[1]{\frac{#1}{2}}
\newcommand{\cl}[1]{\widetilde{#1}}
\newcommand{\iif}{&\quad&\text{if }}
\newcommand{\other}{&\quad&\text{otherwise}}
\newcommand{\pair}[1]{\langle #1 \rangle}
\newcommand{\resp}{resp.~}
\newcommand{\w}{\mathfrak{w}}

\newtheorem{thm}{Theorem}[section]
\newtheorem{rem}[thm]{Remark}
\newtheorem{des}[thm]{Desideratum}

\title{On the uniqueness of generic representations in an $L$-packet}
\author{Hiraku Atobe}
\date{}
\address{Department of mathematics, Kyoto University, Kitashirakawa-Oiwake-cho, Sakyo-ku, Kyoto, 606-8502, Japan}
\email{atobe@math.kyoto-u.ac.jp}

\begin{document}
\maketitle
\begin{abstract}
In this paper, we give a simple and short proof of the uniqueness of 
generic representations in an $L$-packet
for a quasi-split connected classical group over a non-archimedean local field.
\end{abstract}

\section{Introduction}
Let $G$ be a quasi-split connected reductive group defined over a non-archimedean local field $F$
of characteristic zero.
We denote the center of $G$ by $Z$.
A Whittaker datum for $G$ is a conjugacy class of a pair $\w = (B,\mu)$, 
where $B = TU$ is an $F$-rational Borel subgroup of $G$, 
$T$ is a maximal $F$-torus, $U$ is the unipotent radical of $B$, and 
$\mu$ is a generic character of $U(F)$. 
Here, $T(F)$ acts on $U(F)$ by conjugation, 
and we say that a character $\mu$ of $U(F)$ is generic 
if the stabilizer of $\mu$ in $T(F)$ is equal to $Z(F)$.
Let $\Irr(G(F))$ be the set of equivalence classes of irreducible smooth representations of $G(F)$. 
We say that $\pi \in \Irr(G(F))$ is $\w$-generic if $\Hom_{U(F)}(\pi,\mu) \not=0$.
\par
\vskip 10pt

The local Langlands conjecture predicts a canonical partition 
\[
\Irr(G(F)) = \bigsqcup_{\phi \in \Phi(G)}\Pi_\phi, 
\]
where $\Phi(G)$ is the set of  $L$-parameters of $G$, 
which are $\widehat{G}$-conjugacy classes of admissible homomorphisms
\[
\phi \colon \WD_F \rightarrow {}^LG
\]
from the Weil--Deligne group $\WD_F = W_F \times \SL_2(\C)$ of $F$ 
to the $L$-group ${}^LG = \widehat{G} \rtimes W_F$ of $G$.
Here, $W_F$ is the Weil group of $F$ and $\widehat{G}$ is the Langlands dual group of $G$.
The set $\Pi_\phi$ is called the $L$-packet of $\phi$.
In addition, for a Whittaker datum $\w$ for $G$, 
there would exist a bijection
\[
\iota_\w \colon \Pi_\phi \rightarrow \Irr(\Sch_\phi), 
\]
which satisfies certain character identities (see e.g., \cite[\S 2]{Ka}).
Here, $\Sch_\phi = \pi_0(\cent(\im(\phi),\widehat{G})/Z(\widehat{G})^{W_F})$
is the component group of $\phi$, 
which is a finite group.
\par
\vskip 10pt

As a relationship of $L$-packets and $\w$-generic representations, the following is expected.
\begin{des}\label{unique0}
Let $\phi \in \Phi(G)$.
If $\pi \in \Pi_\phi$ is $\w$-generic, then 
$\iota_\w(\pi)$ is the trivial representation of $\Sch_\phi$.
\end{des}
Desideratum \ref{unique0} asserts that each $\Pi_\phi$ has at most one $\w$-generic representation.
\par
\vskip 10pt

When $G=\GL_n$, the local Langlands conjecture has been established 
by Harris--Taylor \cite{HT}, Henniart \cite{He}, and Scholze \cite{Sc}.
In this case, for any $\phi \in \Phi(\GL_n)$, the component group $\Sch_\phi$ is trivial, so that
$\Pi_\phi$ is a singleton.
\par

When $G$ is a quasi-split classical group, i.e., 
$G$ is a symplectic, special orthogonal or unitary group, 
the local Langlands conjecture is almost completely known 
by the recent works of Arthur \cite{Ar} and Mok \cite{Mo}.
In this case, Desideratum \ref{unique0} is 
a special case of the local Gan--Gross--Prasad conjecture \cite[Theorem 17.3]{GGP}.
This conjecture, at least for tempered $L$-parameters, has been proven by 
Waldspurger \cite{W1}, \cite{W2}, \cite{W3}, \cite{W4}, 
Beuzart-Plessis \cite{BP1}, \cite{BP2}, \cite{BP3}, Gan--Ichino \cite{GI}, and the author \cite{At}.
Hence Desideratum \ref{unique0} has already been established, 
but it is proven by long and complicated arguments after 
the results of Arthur \cite{Ar}, Mok \cite{Mo} and 
Kaletha--M\'{i}nguez--Shin--White \cite{KMSW}.
Before these results,  
Jiang and Soudry \cite{JS} have shown 
Desideratum \ref{unique0} for $G=\SO(2n+1)$ by a different method
proving the local converse theorem.
The proof of the local converse theorem for $\SO(2n+1)$ highly relies on 
the theory of the local descent, which makes the proof difficult.
\par
\vskip 10pt

The purpose of this paper is to give another proof of Desideratum \ref{unique0} 
for quasi-split classical groups.
Surprisingly, Desideratum \ref{unique0} is a formal consequence 
from results of Arthur \cite{Ar} and Mok \cite{Mo}, 
so that our proof is much shorter and simpler than before.
In our proof, we use two statements:
One is an intertwining relation (Desideratum \ref{IR}), 
which is a relation between a normalized self-intertwining operator
on an induced representation and the local Langlands correspondence.
The other is Shahidi's result (Theorem \ref{OR}), 
which describes the action of the intertwining operator 
on a canonical Whittaker functional on the induced representation.
Both of them focus on induced representations, 
so that one may seem that they cannot be applied to discrete series representations.
The idea of our proof is as follows:
For a given tempered $\w$-generic representation $\pi$ of $G(F)$, 
take a representation $\tau$ of $\GL_k(E)$ with an extension $E/F$, 
and consider the induced representation $\Ind_{P(F)}^{G'(F)}(\tau \boxtimes \pi)$
of a bigger group $G'(F)$ of the same type as $G(F)$.
Then we can apply two statements to this induced representation. 
Taking several representations $\tau$ and considering 
the associated induced representations of several bigger groups, 
we obtain Desideratum \ref{unique0} for $\pi$.
\par
\vskip 10pt

In this paper, we do not treat the existence of $\w$-generic representations, 
but we only state a desideratum for tempered $L$-parameters.
The following are expected and called Shahidi's conjecture \cite[Conjecture 9.4]{Sh2}.
\begin{des}\label{exist}
Let $\phi \in \Phi(G)$ be a tempered $L$-parameter (see \S \ref{Lpara} below).
Then $\Pi_\phi$ has a $\w$-generic representation. 
\end{des}

When $G=\GL_n$, it follows from a result of Zelevinsky \cite[Theorem 9.7]{Z}.
When $G$ is a quasi-split classical group, 
Arthur \cite[Proposition 8.3.2]{Ar} and Mok \cite[Corollary 9.2.4]{Mo} proved Desideratum \ref{exist} 
by using a global argument.
In addition, Kaletha \cite[Theorem 3.3]{Ka} showed that $\Pi_\phi$ has a $\w'$-generic representation 
for any tempered $L$-parameter $\phi$ of $G$ and any Whittaker datum $\w'$ for $G$.
\par

Before these results of Arthur, Mok and Kaletha, there is a result of Konno \cite{Ko}.
Under the assumption that the residual characteristic of $F$ is not two, 
Konno \cite[Theorem 4.1]{Ko} proved a twisted analogue of 
a result of Moeglin--Waldspurger \cite{MW}, which states that
the `leading' coefficients in the (twisted) character expansion of 
an irreducible representation $\pi$ of $G(F)$ at the identity element 
give the dimensions of certain spaces of degenerate Whittaker models.
By comparing the twisted character expansion of $\GL_n$ 
with the ordinary one of a classical group $G$, 
Konno concluded that 
for each tempered $L$-parameter $\phi$ of $G$, 
$\Pi_\phi$ has a $\w'$-generic representation for some Whittaker datum $\w'$ for $G$
(\cite[Thorem 8.4]{Ko}).
As is mentioned in the remark after Proposition 8.3.2 in \cite{Ar}, 
one may seem to be able to prove the uniqueness (Desideratum \ref{unique0}) 
by the same method.
However it is not so immediately.
To prove the uniqueness by Konno's method, 
one would have to show that the Fourier transforms of the (twisted) orbital integrals of 
$\GL_n$ and $G$
may be distinguished simultaneously 
by a function $f$ on $\GL_n$ and its transfer $f^G$ on $G$, respectively.
\par
\vskip 10pt

Finally, we remark on the archimedean case.
When $F$ is an archimedean local field, 
the local Langlands conjecture has been established by Langlands \cite{L} himself.
Desiderata \ref{unique0} and \ref{exist} were proven by \cite{V}, \cite{Kos} and \cite{She}.
\par

\subsection*{Acknowledgments}
The author is grateful to Professor Atsushi Ichino for his helpful comments. 
Thanks are also due to the referees for helpful comments.
This work was supported by the Foundation for Research Fellowships of Japan Society for the Promotion of Science for Young Scientists (DC1) Grant 26-1322. 
\par

\subsection*{Notations}
Let $F$ be a non-archimedean local field with characteristic zero.
For a finite extension $E/F$, 
we denote by 
$W_E$ and $\WD_E=W_E\times\SL_2(\C)$ 
the Weil and Weil--Deligne groups of $E$, respectively.
The normalized absolute value on $E$ is denoted by $|\cdot|_E$.

\section{Local Langlands correspondence for classical groups}
The local Langlands correspondence (the local Langlands conjecture) 
for quasi-split classical groups
has been established by Arthur \cite{Ar} and Mok \cite{Mo}  
under some assumption on the stabilization of twisted trace formulas.
For this assumption, see also the series of papers
\cite{Stab1}, \cite{Stab2}, \cite{Stab3}, \cite{Stab4}, \cite{Stab5}, 
\cite{Stab6}, \cite{Stab7}, \cite{Stab8}, \cite{Stab9} and \cite{Stab10}.
In this section, we summarize some of its properties which are used in this paper.

\subsection{Generic representations}
Let $G$ be a quasi-split (connected) classical group, i.e., 
$G$ is a unitary, symplectic or special orthogonal group.
We denote by $Z$ the center of $G$.
A Whittaker datum for $G$ is a conjugacy class of a pair $\w = (B,\mu)$, 
where $B = TU$ is an $F$-rational Borel subgroup of $G$ and 
$\mu$ is a generic character of $U(F)$. 
Here, $T(F)$ has the adjoint action on $U(F)$ and so that 
$T(F)$ acts on the set of characters of $U(F)$.
We say that a character $\mu$ of $U(F)$ is generic 
if the stabilizer of $\mu$ in $T(F)$ is equal to $Z(F)$.
\par

Let $\Irr(G(F))$ be the set of equivalence classes of irreducible smooth representations of $G(F)$, 
and $\Irr_\temp(G(F))$ be the subset of $\Irr(G(F))$ 
consisting of classes of irreducible tempered representations.
For a Whittaker datum $\w=(B,\mu)$ with $B=TU$, 
we say that $\pi \in \Irr(G(F))$ is $\w$-generic if
\[
\Hom_{U(F)}(\pi,\mu) \not=0.
\]
\par

\subsection{$L$-parameters and component groups}\label{Lpara}
Let $G$ be a quasi-split (connected) classical group.
If $G$ is a unitary group, 
we denote by $E$ the splitting field of $G$, which is a quadratic extension of $F$.
If $G$ is a symplectic group or a special orthogonal group, we set $E=F$.
We denote by $\widehat{G}$ the Langlands dual group of $G$, and 
by ${}^L G = \widehat{G} \rtimes W_F$ the $L$-group of $G$.
An $L$-parameter $\varphi$ of $G$ is a $\widehat{G}$-conjugacy class of
an admissible homomorphism
\[
\varphi \colon \WD_F \rightarrow {}^LG.
\]
There is a standard representation ${}^L G \rightarrow \GL_N(\C)$ for a suitable $N$ 
(see \cite[\S 7]{GGP}).
By composing an $L$-parameter $\varphi$ with this map, 
we obtain a (conjugate) self-dual representation $\phi$ of $\WD_E$.
More precisely, see \cite[\S 3 and \S8]{GGP}.
We denote by $\Phi(G)$ the set of equivalence classes of (conjugate) self-dual 
representations of $\WD_E$ with suitable type and determinant as in \cite[Theorem 8.1]{GGP}.
Then $\varphi \mapsto \phi$ gives a surjective map
\[
\{\text{$L$-parameters of $G$}\}
\rightarrow \Phi(G),
\]
which is bijective unless $G=\SO(2n)$ in which case, 
the number of each fiber of this map is one or two.
\par

We say that $\phi \in \Phi(G)$ is tempered if $\phi(W_E)$ is bounded.
We denote by $\Phi_\temp(G)$ the subset of $\Phi(G)$
consisting of classes of tempered representations.
\par

If $\phi \in \Phi(G)$ is given by an $L$-parameter $\varphi$, 
we define the component group $\Sch_\phi$ of $\phi$ by
\[
\Sch_\phi = \pi_0(\cent(\im(\varphi), \widehat{G})/Z(\widehat{G})^{W_F}).
\]
Here, $Z(\widehat{G})$ is the center of $\widehat{G}$.
Note that 
$\Sch_\phi$ does not depend on the choice of $\varphi$, and
$\Sch_\phi$ is isomorphic to $(\Z/2\Z)^r$ for some non-negative integer $r$.
As in \cite[\S 4]{GGP}, $\Sch_\phi$ is described explicitly as follows.
We denote by $\BB_\phi$ the set of equivalence classes of 
representations $\phi'$ of $\WD_E$ such that
\begin{itemize}
\item
$\phi'$ is contained in $\phi$; 
\item
$\phi'$ is a multiplicity-free sum of
irreducible (conjugate) self-dual representations of $\WD_E$ with the same type as $\phi$.
\end{itemize}
Also we put
\[
\BB_\phi^+=
\left\{
\begin{aligned}
&\BB_\phi	\iif \text{$G=\U(m)$ or $G=\SO(2n+1)$},\\
&\{\phi' \in \BB_\phi\ |\ \dim(\phi') \in 2\Z\} \iif \text{$G=\Sp(2n)$ or $G=\SO(2n)$}.
\end{aligned}
\right.
\]
When $\phi_0$ is an irreducible representation of $\WD_E$, 
the multiplicity of $\phi_0$ in $\phi$ is denoted by $m(\phi_0;\phi)$, i.e.,
\[
m(\phi_0;\phi) = \dim\Hom(\phi_0,\phi).
\]
For $\phi_1, \phi_2 \in \BB_\phi$, 
we write $\phi_1 \sim \phi_2$ if 
\[
m(\phi_0;\phi_1) + m(\phi_0; \phi_2) \equiv m(\phi_0;\phi) \bmod2
\]
for any irreducible (conjugate) self-dual representation $\phi_0$ of $\WD_E$ 
with the same type as $\phi$.
For $a \in \Sch_\phi$, 
choose a semisimple representative $s \in \cent(\im(\varphi), \widehat{G})$ of $a$.
We regard $s$ as an automorphism on the space of $\phi$, 
which commutes with the action of $\WD_E$.
We denote by $\phi^{s=-1}$ the $(-1)$-eigenspace of $s$, which is a representation of $\WD_E$.
We define $\phi^a \in \BB_\phi^+$ so that 
\[
m(\phi_0;\phi^a) \equiv m(\phi_0;\phi^{s=-1}) \bmod 2 
\]
for any irreducible (conjugate) self-dual representation $\phi_0$ of $\WD_E$ 
with the same type as $\phi$.
Then the image of $\phi^a$ in $\BB_\phi^+/\sim$ does not depend on the choice of $s$, 
and the map 
\[
\Sch_\phi \rightarrow \BB_\phi^+/\sim,\ a \mapsto \phi^a
\]
is bijective.

\subsection{Local Langlands correspondence for classical groups}
In this subsection, we introduce $\Pi(G)$, which is a quotient of $\Irr(G(F))$,  and 
state some expected properties of the local Langlands correspondence which we need.
\par

First, we consider an even special orthogonal group $G=\SO(2n)$.
Choose $\ep \in \orth(2n, F) \setminus \SO(2n, F)$. 
For $\pi \in \Irr(\SO(2n, F))$, 
its conjugate $\pi^\ep$ is defined by $\pi^\ep(h)=\pi(\ep^{-1}h\ep)$ for $h \in \SO(2n, F)$.
We define an equivalence relation $\sim_\ep$ on $\Irr(\SO(2n, F))$ by 
\[
\pi \sim_{\ep} \pi^\ep
\]
for $\pi \in \Irr(\SO(2n, F))$. 
In \cite{Ar}, one has parametrized not $\Irr(\SO(2n, F))$ but $\Irr(\SO(2n, F))/\sim_\ep$.
Note that $\pi$ is tempered (\resp $\w$-generic) if and only if so is $\pi^\ep$.
\par

We return the general setting.
Let $G$ be a quasi-split (connected) classical group.
We define $\Pi(G)$ by
\[
\Pi(G)=
\left\{
\begin{aligned}
&\Irr(G(F))/\sim_{\ep} \iif \text{$G=\SO(2n)$},\\
&\Irr(G(F))	\other.
\end{aligned}
\right.
\]
For $\pi \in \Irr(G(F))$, we denote the image of $\pi$ under the canonical map 
$\Irr(G(F)) \rightarrow \Pi(G)$ by $[\pi]$.
We say that $[\pi] \in \Pi(G)$ is $\w$-generic (\resp tempered) if so is any representative $\pi$.
Also, we put $\Pi_\temp(G)$ to be the image of $\Irr_\temp(G(F))$ in $\Pi(G)$.
\par

Now we are ready to describe desiderata for the local Langlands correspondence.
\begin{des}\label{des}
Let $G$ be a quasi-split (connected) classical group.
We fix a Whittaker datum $\w$ for $G$.
\begin{enumerate}
\item
There exists a canonical surjection (not depending on $\w$)
\[
\Pi(G) \rightarrow \Phi(G).
\]
For $\phi \in \Phi(G)$, we denote by $\Pi_\phi$
the inverse image of $\phi$ under this map,
and call $\Pi_\phi$ the $L$-packet of $\phi$.
\item
There exists a bijection (depending on $\w$)
\[
\iota_\w \colon \Pi_\phi \rightarrow \widehat{\Sch_\phi},
\]
which satisfies certain character identities. 
Here, $\widehat{\Sch_\phi}$ is the Pontryagin dual of $\Sch_\phi$.
\item
We have 
\[
\Pi_\temp(G) = \bigsqcup_{\phi \in \Phi_\temp(G)}\Pi_\phi.
\]
\item
Assume that $\phi=\phi_\tau \oplus \phi_0 \oplus {}^c\phi_\tau^\vee$, where 
\begin{itemize}
\item
$\phi_0$ is an element in $\Phi_\temp(G_0)$ with a classical group $G_0$ of the same type as $G$;
\item
$\phi_\tau$ is a tempered representation of $\WD_E$ with dimension $k$.
\end{itemize}
Here, ${}^c\phi_\tau$ is the Galois conjugate of $\phi_\tau$ (see \cite[\S 3]{GGP}).
Let $\tau$ be the irreducible tempered representation of $\GL_k(E)$ corresponding to $\phi_\tau$.
Then 
for a representative $\pi_0$ of an element in $\Pi_{\phi_0}$, 
the induced representation 
\[
\Ind_{P(F)}^{G(F)}(\tau \otimes \pi_0)
\]
decomposes into a direct sum of irreducible tempered representations of $G(F)$,
where $P=M_PU_P$ is a parabolic subgroup of $G$ 
with Levi subgroup $M_P(F)=\GL_k(E) \times G_0(F)$. 
The $L$-packet $\Pi_\phi$ is given by
\[
\Pi_\phi = \{[\pi]\ |\ \pi \subset \Ind_{P(F)}^{G(F)}(\tau \otimes \pi_0), [\pi_0] \in \Pi_{\phi_0}\}.
\]
Moreover there is a canonical inclusion $\Sch_{\phi_0} \hookrightarrow \Sch_\phi$.
If $\pi \subset \Ind_{P(F)}^{G(F)}(\tau \otimes \pi_0)$ and 
$\w_0$ is the Whittaker datum for $G_0$ given by the restriction of $\w$, 
then $\iota_\w([\pi])|\Sch_{\phi_0} = \iota_{\w_0}([\pi_0])$.
\item
Assume that 
\[
\phi = \phi_1|\cdot|^{s_1} \oplus \dots \oplus \phi_r|\cdot|^{s_r} \oplus 
\phi_0 \oplus {}^c\phi_r^\vee|\cdot|^{-s_r} \oplus \dots \oplus {}^c\phi_1^\vee|\cdot|^{-s_1}
\]
where 
\begin{itemize}
\item
$\phi_0$ is an element in $\Phi_\temp(G_0)$ with a classical group $G_0$ of the same type as $G$;
\item
$\phi_{i}$ is a tempered representation of $\WD_E$ with dimension $k_i$ 
for $1 \leq i \leq r$; 
\item
$s_i$ is a real number such that $s_1 \geq \dots \geq s_r>0$.
\end{itemize}
Let $\tau_i$ be the irreducible tempered representation of $\GL_{k_i}(E)$ 
corresponding to $\phi_{i}$.
Then the $L$-packet $\Pi_\phi$ consists of (the equivalent classes of)
the unique irreducible quotient $\pi$ of 
the standard module
\[
\Ind_{P(F)}^{G(F)}
(\tau_1|\det|_E^{s_1} \otimes \dots \otimes \tau_r|\det|_E^{s_r} \otimes \pi_0), 
\]
where $\pi_0$ runs over (representatives of) elements in $\Pi_{\phi_0}$.
Here, $P=M_PU_P$ is a parabolic subgroup of $G$ with Levi subgroup
$M_P(F) = \GL_{k_1}(E) \times \dots \times \GL_{k_r}(E) \times G_0(F)$.
Moreover there is a canonical inclusion $\Sch_{\phi_0} \hookrightarrow \Sch_\phi$, 
which is in fact bijective. 
If $\pi$ is the unique irreducible quotient of the above standard module
and $\w_0$ is the Whittaker datum for $G_0$ given by the restriction of $\w$, 
then $\iota_\w([\pi])|\Sch_{\phi_0} = \iota_{\w_0}([\pi_0])$.
\end{enumerate}
\end{des}

\begin{rem}
Arthur \cite{Ar} and Mok \cite{Mo} have established the local Langlands correspondence
for tempered parameters, i.e., Desideratum \ref{des} (2), (3) and (4).
Using (5) and the Langlands classification, we may obtain (1).
\end{rem}

\subsection{Intertwining relations}\label{sec.IR}
The purpose of this paper is to give a simple and short proof of Desideratum \ref{unique0}
when $G$ is a quasi-split classical group.
To do this, we need one more technical desideratum.
This is a relation between a normalized self-intertwining operator
and the local Langlands correspondence.
\par

Let $G$ be a quasi-split (connected) classical group
and $\w = (B,\mu)$ be a Whittaker datum for $G$.
We denote the unipotent radical of $B$ by $U$.
Fix a positive integer $k$, and put 
$G'= \U(m+2k)$, $\Sp(2n+2k)$ or $\SO(m+2k)$ 
when $G= \U(m)$, $\Sp(2n)$ or $\SO(m)$, respectively.
If $G=\Sp(2n)$ or $G=\SO(2n)$, we assume that $k$ is even.
Let $\w' = (B', \mu')$ be a Whittaker datum for $G'$ such that 
$B= B' \cap G$ and $\mu = \mu'|U(F)$.
We consider a maximal $F$-parabolic subgroup $P=M_PU_P$ of $G'$ containing $B'$
such that the Levi subgroup $M_P$ of $P$ is of the form $M_P(F) \cong \GL_k(E) \times G(F)$. 
Here $U_P$ is the unipotent radical of $P$, so that 
$U_P$ is contained in the unipotent radical $U'$ of $B'$.
We denote by $\delta_P$ the modulus character of $P$.
Let $\pi$ (\resp $\tau$) be an irreducible tempered representation of $G(F)$ (\resp $\GL_k(E)$) 
on a space $\VV_{\pi}$ (\resp $\VV_\tau$).
We consider the normalized induction
\[
I_0(\tau \boxtimes \pi) = \Ind_{P(F)}^{G'(F)}(\tau \boxtimes \pi), 
\]
which consists of smooth functions $f_0 \colon G'(F) \rightarrow \VV_\tau \otimes \VV_\pi$ 
such that
\[
f_0(uagg') = \delta_P^{\half{1}}(a)(\tau(a) \boxtimes \pi(g)) f_0(g')
\]
for $u \in U_P(F)$, $a \in \GL_k(E)$, $g \in G(F)$ and $g' \in G'(F)$.
\par

We denote by $A_P$ the split component of the center of $M_P$ and
by $W(M_P) = \Norm(A_P, G')/M_P$ the relative Weyl group for $M_P$.
Note that $W(M_P) \cong \Z/2\Z$ (unless $G$ is the split $\SO(2)$ and $k=1$).
Let $w \in W(M_P)$ be the unique non-trivial element 
which induces an automorphism of $M_P(F) \cong \GL_k(E) \times G(F)$
whose restriction on $G(F)$ is trivial.
Fixing a splitting of $G'$, which gives the Whittaker datum $\w'$, 
we obtain a representative $\cl{w} \in G'(F)$ of $w$ as in 
\cite[\S 2.3]{Ar} and \cite[\S 3.3]{Mo}.
\par

Now suppose that $w(\tau \boxtimes \pi) \cong \tau \boxtimes \pi$, where
$w(\tau \boxtimes \pi)(m) = (\tau \boxtimes \pi)(\cl{w}^{-1}m\cl{w})$ for $m \in M_P(F)$.
Then Arthur \cite[\S 2.3]{Ar} and Mok \cite[\S 3.3]{Mo} 
have defined a normalized intertwining operator
\[
R_{\w'}(w, \tau \boxtimes \pi) \colon I_0(\tau \boxtimes \pi) \rightarrow I_0(w(\tau \boxtimes \pi))
\]
which depends on the Whittaker datum $\w'$ for $G'$.
Let $\cl{\tau\boxtimes \pi}(\cl{w}) \colon 
\VV_\tau \otimes \VV_\pi \rightarrow \VV_\tau \otimes \VV_\pi$ be a unique linear map
satisfying
\begin{itemize}
\item
$\cl{\tau\boxtimes \pi}(\cl{w}) \circ (\tau \boxtimes \pi) (\cl{w}^{-1} m \cl{w})
= (\tau \boxtimes \pi) (m) \circ \cl{\tau\boxtimes \pi}(\cl{w})$;
\item
$\cl{\tau\boxtimes \pi}(\cl{w}) = \cl{\tau}(\cl{w}) \otimes \cl{\pi}(\cl{w})$, 
where $\cl{\pi}(\cl{w}) \colon \VV_\pi \rightarrow \VV_\pi$ is the identity map, and
$\cl{\tau}(\cl{w}) \colon \VV_\tau \rightarrow \VV_\tau$ 
is the unique linear map which preserves a Whittaker functional on $\VV_\tau$
(with respect to the Whittaker datum for $\GL_k(E)$ given by the restriction of $\w'$).
\end{itemize}
Note that $\tau$ is generic since $\tau$ is a tempered representation of $\GL_k(E)$.
Finally, 
we define a normalized self-intertwining operator 
$R_{\w'}(w, \cl{\tau \boxtimes \pi}) \colon I_0(\tau \boxtimes \pi) \rightarrow I_0(\tau \boxtimes \pi)$ by
\[
R_{\w'}(w, \cl{\tau \boxtimes \pi}) = \cl{\tau\boxtimes \pi}(\cl{w}) \circ R_{\w'}(w,\tau \boxtimes \pi).
\]
\par

Suppose that $[\pi] \in \Pi_\phi$ for $\phi \in \Phi_\temp(G)$.
Let $\phi_\tau$ be the tempered representation of $\WD_E$ corresponding to $\tau$.
Note that ${}^c \phi_\tau ^\vee \cong \phi_\tau$ 
since $w(\tau \boxtimes \pi) \cong \tau \boxtimes \pi$ so that $\tau$ is (conjugate) self-dual. 
Put $\phi'= \phi_\tau \oplus \phi \oplus \phi_\tau \in \Phi_\temp(G')$.
Let $\pi'$ be an irreducible direct summand of $I_0(\tau \otimes \pi)$. 
Then we have $[\pi'] \in \Pi_{\phi'}$.
\par

For the proof of Desideratum \ref{unique0}, 
we use the following extra desideratum.
\begin{des}[Intertwining relation]\label{IR}
Let the notation be as above.
Assume that $\phi_\tau \in \BB_{\phi'}^+$, so that
there exists a unique element $a \in \Sch_{\phi'}$ such that 
$\phi'^{a} = \phi_\tau$ in $\BB_{\phi'}^+/\sim$.
Then 
\[
R_{\w'}(w, \cl{\tau \boxtimes \pi}) | \pi' = \iota_{\w'}([\pi'])(a) \cdot \id.
\]
\end{des}

The local Langlands correspondence established by Arthur and Mok
satisfies the intertwining relation.
\begin{thm}\label{IRproof}
The intertwining relation (Desideratum \ref{IR}) follows from 
Theorems 2.2.1 and 2.4.1 in \cite{Ar} when $G=\Sp(2n)$ or $G=\SO(m)$, 
and Theorems 3.2.1 and 3.4.3 in \cite{Mo} when $G=\U(m)$.
\end{thm}
\begin{proof}
We prove only the case when $G=\SO(2n)$.
The other cases are similar.
\par

Suppose that $G=\SO(2n)$.
In this case, 
since $\phi_\tau$ and $\phi$ are orthogonal representations, 
there are non-degenerated symmetric matrices $A$ and $B$ of size $k$ and $2n$ such that
\[
{}^t\phi_\tau(x) A \phi_\tau(x) = A
\quad\text{and}\quad
{}^t\phi(x) B \phi(x) = B
\]
for $x \in \WD_F$, respectively.
We regard $\widehat{G'} = \SO(2n+2k, \C)$ as the special orthogonal group 
with respect to the symmetric matrix $\mathrm{diag}(A, B, -A)$.
The image of $\phi' = \phi_\tau \oplus \phi \oplus \phi_\tau$ is contained in $\orth(2n+2k, \C)$.
Let $\{e'_1, \dots, e'_k, e_1, \dots, e_{2n}, e_1'', \dots,  e_{k}''\}$ 
be the canonical basis of $\C^{2n+2k}$.
Then $\widehat{M}_P$ is realized as the Levi subgroup of $\widehat{G'}$ 
stabilizing two isotypic subspaces
\[
\mathrm{span}\{e'_i + e_{i}''\ |\ i=1, \dots, k\}
\quad\text{and}\quad
\mathrm{span}\{e'_i - e_{i}''\ |\ i=1, \dots, k\}.
\]
Note that the image of $\phi' = \phi_\tau \oplus \phi \oplus \phi_\tau$ stabilizes these two subspaces.
\par

Let $u \in \widehat{G'}$ be the element which acts on 
$\{e'_1, \dots, e'_k, e_1, \dots, e_{2n}\}$ by $+1$,
and on $\{e''_{1}, \dots, e''_k\}$ by $-1$.
Then $u \in \cent(\im(\phi'), \widehat{G'}) \cap \Norm(\widehat{M}_P, \widehat{G'})$.
Note that $a \in \Sch_{\phi'}$ and $w \in W(M_P) \cong W(\widehat{M}_P)$ are
the images of $u$ under the canonical maps $\cent(\im(\phi'), \widehat{G'}) \rightarrow \Sch_{\phi'}$ 
and $\Norm(\widehat{M}_P, \widehat{G'}) \rightarrow W(\widehat{M}_P)$, respectively.
By applying Theorems 2.2.1 (b) and 2.4.1 in \cite{Ar} to $u$, 
we obtain the character identity
\[
\sum_{[\pi'] \in \Pi_{\phi'}}\iota_{\w'}([\pi'])(a) \cdot \tr(\pi'(f'))
=
\sum_{[\pi] \in \Pi_{\phi}}\iota_{\w}(\cl{\pi})(\cl{u}) \cdot 
\tr(R_{\w'}(w, \cl{\tau \boxtimes \pi})I_0(\tau \otimes \pi)(f'))
\]
for any $f' \in \cl{\mathcal{H}}(G')$.
Here, $\cl{\mathcal{H}}(G')$ is the subalgebra of the Hecke algebra $\mathcal{H}(G')$ of $G'$
consisting of $\mathrm{Out}(G')$-invariant functions, 
and $\iota_{\w}(\cl{\pi})(\cl{u}) \in \C^\times$ is a constant.
(In \cite{Ar}, $\iota_{\w'}([\pi'])(a)$ and $\iota_{\w}(\cl{\pi})(\cl{u})$ are 
denoted by $\pair{x,\pi'}$ and $\pair{\cl{u}, \cl{\pi}}$, respectively.)
\par

The constant $\iota_{\w}(\cl{\pi})(\cl{u})$ is defined as follows:
Put 
\[
\mathfrak{N}_{\phi'} = 
\pi_0( \cent(\im(\phi'), \widehat{G'}) \cap \Norm(\widehat{M}_P, \widehat{G'})/Z(\widehat{G'})^{W_F}).
\] 
Then $\mathfrak{N}_{\phi'}$ contains $\Sch_{\phi}$
(see also the diagram (2.4.3) in \cite{Ar}).
We set $\cl{u} \in \mathfrak{N}_{\phi'}$ to be the image of $u$.
Since $k$ is even, 
for $u' \in \Norm(\widehat{M}_P, \widehat{G'})$, 
the action of $u'$ on $\{e_1, \dots, e_{2n}\}$ gives an element $u'_0 \in \widehat{G}$.
The map $u' \mapsto u'_0$ induces a section
\[
s' \colon \mathfrak{N}_{\phi'} \rightarrow \Sch_{\phi}.
\]
We define the map
$\iota_{\w}(\cl{\pi}) \colon \mathfrak{N}_{\phi'} \rightarrow \{\pm1\}$
by 
\[
\iota_{\w}(\cl{\pi}) = \iota_{\w}([\pi]) \circ s'.
\]
In particular, since $s'(\cl{u})=1$, we have $\iota_{\w}(\cl{\pi})(\cl{u}) = \iota_{\w}([\pi])(1) = 1$.
\par

Hence the character identity implies that
\[
\sum_{[\pi'] \in \Pi_{\phi'}}\iota_{\w'}([\pi'])(a) \cdot \tr(\pi'(f'))
=
\sum_{[\pi] \in \Pi_{\phi}}\sum_{\pi' \subset I_0(\tau \otimes \pi)}
\tr(R_{\w'}(w, \cl{\tau \boxtimes \pi})\pi'(f'))
\]
for $f' \in \cl{\mathcal{H}}(G')$.
Therefore we have $R_{\w'}(w, \cl{\tau \boxtimes \pi}) | \pi' = \iota_{\w'}([\pi'])(a) \cdot \id$.
\end{proof}

\begin{rem}
\begin{enumerate}
\item
If $G=\SO(2n)$ but $k$ were odd, 
there would be no canonical choice of $\cl{\pi}(\cl{w})$. 
In this case, for each choice, the constant
$\iota_{\w}(\cl{\pi})(\cl{u}) \in \C^\times$ is defined by
using the pairing of \cite[Theorem 2.2.4]{Ar}.
\item
When $G$ is a pure inner form of a quasi-split unitary group, 
the definition of the section $s' \colon \mathfrak{N}_{\phi'} \rightarrow \Sch_{\phi}$ 
is slightly more complicated.
See \cite[\S 2.4.1]{KMSW}.
\end{enumerate}
\end{rem}

\section{Proof of Desideratum \ref{unique0} for classical groups}
Now we give a proof of Desideratum \ref{unique0}
when $G$ is a quasi-split classical group.
Namely: 
\begin{thm}\label{unique}
Let $G$ be a quasi-split classical group. 
Assume Desiderata \ref{des} and \ref{IR} for $G$.
For $\phi \in \Phi(G)$, 
if $[\pi] \in \Pi_\phi$ is $\w$-generic, then
$\iota_\w([\pi])$ is the trivial representation of $\Sch_\phi$.
\end{thm}

The proof of Theorem \ref{unique} is a formal consequence of
the intertwining relation (Desideratum \ref{IR}) together with Shahidi's result.
We recall Shahidi's result in \S \ref{canonical}.

\subsection{Canonical Whittaker functional}\label{canonical}
In this subsection, we recall canonical Whittaker functionals of induced representations.
\par

Let $G$ be a quasi-split (connected) classical group
and $\w = (B,\mu)$ be a Whittaker datum for $G$.
Fix a positive integer $k$.
If $G=\Sp(2n)$ or $G=\SO(2n)$, we assume that $k$ is even.
Let $G'$, $\w' = (B', \mu')$, $P=M_PU_P$, $w \in W(M_P)$ and $\cl{w} \in G'(F)$ 
be as in \S \ref{sec.IR}.
Let $\tau$ be an irreducible tempered representation of $\GL_k(E)$ on a space $\VV_\tau$.
For $s \in \C$, we realize $\tau_s = \tau|\det|_E^s$ on $\VV_\tau$ by setting
$\tau_s(a)v \coloneqq |\det(a)|_E^s \tau(a)v$ for $v \in \VV_\tau$ and $a \in \GL_k(E)$.
Let $\pi$ be an irreducible tempered representation of $G(F)$ on a space $\VV_{\pi}$.
We consider the normalized induction
\[
I_s(\tau \boxtimes \pi) = \Ind_{P(F)}^{G'(F)}(\tau_s \boxtimes \pi).
\]
\par

Now we assume that $\pi$ is $\w$-generic.
We regard $\w'$ as a Whittaker datum for $M_P$ by the restriction.
Since $\tau$ is tempered, we see that $\tau \boxtimes \pi$ is $\w'$-generic.
Let $\omega \colon \VV_\tau \otimes \VV_\pi \rightarrow \C$ 
be a nonzero $\w'$-Whittaker functional, i.e., 
\[
\omega((\tau \boxtimes \pi)(u)v) = \mu'(u)\omega(v)
\]
for $u \in U' \cap M_P$ and $v \in \VV_\tau \otimes \VV_\pi$.
Note that the representative $\cl{w} \in G'(F)$ of $w \in W(M_P)$ and the linear map 
$\cl{\tau\boxtimes \pi}(\cl{w}) \colon 
\VV_{\tau} \otimes \VV_\pi \rightarrow \VV_{\tau} \otimes \VV_\tau$
satisfy that
\begin{itemize}
\item
$\cl{w}^{-1}B'\cl{w} = B'$;
\item
$\mu'(\cl{w}^{-1}u\cl{w}) = \mu'(u)$ for any $u \in U'(F) \cap M_P(F)$;
\item
$\omega \circ \cl{\tau\boxtimes \pi}(\cl{w}) = \omega$.
\end{itemize}
See also \cite[\S 2.5]{Ar} and \cite[\S 3.5]{Mo}.
\par

We define the Jacquet integral 
$\WW_{\mu',\omega}(g', f_s)$ for $f_s \in I_s(\tau \boxtimes \pi)$ by
\[
\WW_{\mu', \omega}(g',f_s) = \int_{U_P(F)}\omega(f_s(\cl{w}^{-1}ug')) \mu'(u)^{-1}du, 
\]
where $du$ is a Haar measure on $U_P(F)$.
By \cite[Proposition 2.1]{CS} and \cite[Proposition 3.1]{Sh1}, 
the integral $\WW_{\mu',\omega}(g', f_s)$ 
is absolutely convergent for $\re(s) \gg 0$, and
has an analytic continuation as an entire function of $s \in \C$.
The map $f_0 \mapsto \WW_{\mu', \omega}(1,f_0)$ gives a nonzero $\w'$-Whittaker functional
\[
\Omega_{\mu',\omega} \in \Hom_{U'(F)}(I_0( \tau \boxtimes \pi), \mu').
\]
\par

The following theorem follows from Shahidi's results \cite{Sh1}, \cite{Sh2}.
See also \cite[Theorem 2.5.1]{Ar} and \cite[Proposition 3.5.3]{Mo}.

\begin{thm}\label{OR}
Let $R_{\w'}(w, \cl{\tau \boxtimes \pi}) \colon I_0(\tau \boxtimes \pi) \rightarrow I_0(\tau \boxtimes \pi)$
be the normalized self-intertwining operator and
$\Omega_{\mu',\omega} \colon I_0(\tau \boxtimes \pi) \rightarrow \C$ 
be the $\w'$-Whittaker functional defined as above.
Then we have
\[
\Omega_{\mu',\omega} \circ R_{\w'}(w,\cl{\tau \boxtimes \pi}) = \Omega_{\mu',\omega}.
\]
\end{thm}
Recall that to define $\Omega_{\mu',\omega}$, we have to choose
a Haar measure $du$ on $U_P(F)$.
We observe that Theorem \ref{OR} is independent of this choice.

\subsection{Proof}\label{proof}
Now we prove Theorem \ref{unique}.
First, we consider the tempered case.
Let $\phi \in \Phi_\temp(G)$.
Suppose that $[\pi] \in \Pi_\phi$ is $\w$-generic.
Fix a non-trivial element $a \in \Sch_\phi$.
It is enough to show that $\iota_\w([\pi])(a) = 1$.
Choose a representative $\phi_\tau \in \BB_\phi^+$ of $\phi^{a} \in \BB_\phi^+/\sim$. 
Then we have ${}^c \phi_\tau ^\vee \cong \phi_\tau$.
Let $\tau \in \Irr(\GL_k(E))$ be the irreducible tempered representation corresponding to $\phi_\tau$, 
where $k = \dim(\phi_\tau)$.
Note that $k$ is even if $G=\Sp(2n)$ or $G=\SO(2n)$.
Let $G'$ and $\w'$ be as in \S \ref{canonical}. 
Put
\[
\phi'= \phi_\tau \oplus \phi \oplus \phi_\tau \in \Phi_\temp(G').
\]
Then the canonical inclusion $\Sch_{\phi} \hookrightarrow \Sch_{\phi'}$ is in fact bijective 
(cf.~\S \ref{Lpara}).
Consider the induced representation
\[
\pi'=I_0(\tau \boxtimes \pi).
\]
Then by Desideratum \ref{des} (4), we see that
$\pi'$ is irreducible, $[\pi'] \in \Pi_{\phi'}$ and $\iota_{\w'}([\pi'])|\Sch_\phi = \iota_{\w}([\pi])$.
Moreover, by Desideratum \ref{IR}, we have
\[
R_{\w'}(w, \cl{\tau \boxtimes \pi}) = \iota_{\w'}([\pi'])(a) \cdot \id.
\]
On the other hand, by Theorem \ref{OR}, we have
\[
\Omega_{\mu',\omega} \circ R_{\w'}(w, \cl{\tau \boxtimes \pi}) = \Omega_{\mu',\omega}.
\]
Since $\Omega_{\mu',\omega}$ is a nonzero functional on $\pi'$, we must have 
\[
\iota_{\w'}([\pi'])(a) =1.
\]
This implies that $\iota_\w([\pi])(a) = 1$.
\par

Next, we consider the general case.
Let $\phi \in \Phi(G)$ and assume that $[\pi] \in \Pi_\phi$ is $\w$-generic.
We may decompose 
\[
\phi = \phi_1|\cdot|^{s_1} \oplus \dots \oplus \phi_r|\cdot|^{s_r} \oplus 
\phi_0 \oplus {}^c\phi_r^\vee|\cdot|^{-s_r} \oplus \dots \oplus {}^c\phi_1^\vee|\cdot|^{-s_1}
\]
as in Desideratum \ref{des} (5).
Then $\pi$ is the unique Langlands quotient of
\[
\Ind_{P(F)}^{G(F)}
(\tau_1|\det|_E^{s_1} \otimes \dots \otimes \tau_r|\det|_E^{s_r} \otimes \pi_0)
\]
for some $\pi_0 \in \Irr_\temp(G_0(F))$ such that $[\pi_0] \in \Pi_{\phi_0}$.
Here $\tau_i$ is the irreducible tempered representation of $\GL_{k_i}(E)$ 
corresponding to $\phi_i$.
Moreover, we have $\iota_\w([\pi])|\Sch_{\phi_0} = \iota_{\w_0}([\pi_0])$
for the Whittaker datum $\w_0=(B \cap G_0, \mu|U_0(F))$ for $G_0$, 
where $U_0 = U \cap G_0$ is the unipotent radical of 
the Borel subgroup $B_0 = B \cap G_0$ of $G_0$.
By a result of Rodier \cite{R} and \cite[Corollary 1.7]{CS}, there is an isomorphism
\[
\Hom_{U(F)}(
\Ind_{P(F)}^{G(F)}(\tau_1|\det|_E^{s_1} \otimes \dots \otimes \tau_r|\det|_E^{s_r} \otimes \pi_0)
,\mu) \cong \Hom_{U_0(F)}(\pi_0,\mu|U_0(F)).
\]
However, since $\pi$ is $\w$-generic, by the standard module conjecture proved in
\cite{CShahidi}, \cite{HM} and \cite{HO}, 
the standard module
$\Ind_{P(F)}^{G(F)}(\tau_1|\det|_E^{s_1} \otimes \dots \otimes \tau_r|\det|_E^{s_r} \otimes \pi_0)$
is irreducible, so that it is isomorphic to $\pi$.
This implies that $\pi_0$ is $\w_0$-generic.
Hence $\iota_{\w_0}([\pi_0])$ is the trivial character of $\Sch_{\phi_0}$.
Since the inclusion $\Sch_{\phi_0} \hookrightarrow \Sch_\phi$ is bijective, 
we see that $\iota_\w([\pi])$ is the trivial character of $\Sch_{\phi}$.
This completes the proof of Theorem \ref{unique}.

\subsection{Remark}
Finally, we remark on \cite[Lemma 2.5.5]{Ar}.
This lemma is the ``converse'' of our proof of Theorem \ref{unique}.
Roughly speaking, 
this lemma asserts that when the residual characteristic of $F$ is not two, 
the uniqueness of generic representations in an $L$-packet
(Theorem \ref{unique}) for all proper Levi subgroups $M$ of $G$ implies that
\cite[Theorem 2.4.1]{Ar}, which we use to prove
the intertwining relation for $G$ (Theorem \ref{IRproof}).
On the other hand, to prove Theorem \ref{unique} for $G$, 
we used the intertwining relation for a bigger group $G'$, 
which has a Levi subgroup $M_P(F) \cong \GL_k(E) \times G(F)$.
Therefore, in the cases when this lemma is used, 
one should give another proof of Theorem \ref{unique}  without the intertwining relation.
\par

Arthur and Mok have applied \cite[Lemma 2.5.5]{Ar} 
only to the basic cases \cite[Lemmas 6.4.1, 6.6.2]{Ar} and \cite[Proposition 7.4.3]{Mo},
when the Levi subgroup $M$ is a torus or a product of torus and $\SL(2)$.
As noted in the proof of \cite[Lemma 6.4.1]{Ar}, 
Theorem \ref{unique} has already been established for these basic cases.
For tori, it is trivial since the component groups are always trivial.
Hence one only treats $G=\SL(2)$.
More precisely, for $\phi \in \Phi_\temp(\SL(2))$ and a Whittaker datum $\w$ for $\SL(2,F)$, 
one has to check that
$\pi \in \Pi_\phi$ is $\w$-generic if and only if 
$\iota_\w(\pi)$ is the trivial representation of $\Sch_\phi$.
This has been shown by Kottwitz--Shelstad (the end of \S 5.3 in \cite{KS}).
See also \cite[\S 2]{LL}.


\end{document}